\documentclass[12pt,leqno,a4paper] {amsart}
\usepackage{amssymb,enumerate}
\newcommand{\FF}{{\mathbb{F}}}

\newcommand{\bG} {\mathbf G}
\newcommand{\bT} {\mathbf T}

\newcommand{\cB} {\mathcal B}

\newcommand{\fA} {\mathfrak A}
\newcommand{\fS} {\mathfrak S}

\newcommand{\IBr}{{{\operatorname{IBr}}}}
\newcommand{\Irr}{{{\operatorname{Irr}}}}
\newcommand{\reg}{{{\operatorname{reg}}}}
\newcommand{\SC}{{{\operatorname{sc}}}}

\newcommand{\GL}{\operatorname{GL}}
\newcommand{\PSL}{\operatorname{PSL}}
\newcommand{\SL}{\operatorname{SL}}

\newcommand{\SU}{\operatorname{SU}}

\newcommand{\PSp}{\operatorname{PSp}}
\newcommand{\Sp}{\operatorname{Sp}}
\newcommand{\OO}{\operatorname{O}}
\newcommand{\SO}{\operatorname{SO}}
\newcommand{\Spin}{\operatorname{Spin}}
\newcommand{\HSpin}{\operatorname{HSpin}}
\newcommand{\tw}[1]{{}^{#1}\!}

\let\vhi=\varphi
\let\la=\lambda

\newtheorem{thm}{Theorem}[section]
\newtheorem{lem}[thm]{Lemma}
\newtheorem{cor}[thm]{Corollary}
\newtheorem{prop}[thm]{Proposition}

\newtheorem{thmA}{Theorem}
\newtheorem{conjA}[thmA]{Conjecture}

\theoremstyle{definition}
\newtheorem{rem}[thm]{Remark}

\begin{document}

\title{On Willems' conjecture on Brauer character degrees}
\date{\today}
\author{Gunter Malle}
\address{FB Mathematik, TU Kaiserslautern, Postfach 3049,
  67653 Kaisers\-lautern, Germany.}
\email{malle@mathematik.uni-kl.de}



\begin{abstract}
In 2005 Wolfgang Willems put forward a conjecture proposing a lower bound for
the sum of squares of the degrees of the irreducible $p$-Brauer characters of a
finite group $G$.
We prove this conjecture for the prime $p=2$. For this we rely on the recent
reduction of Willems' conjecture to a question on quasi-simple groups by
Tong-Viet. We also verify the conditions of Tong-Viet for certain families of
finite quasi-simple groups and odd primes. On the way we obtain lower bounds for
the number of regular semisimple conjugacy classes in finite groups of Lie type.
\end{abstract}

\thanks{The author gratefully acknowledges support by the SFB TRR 195.}

\maketitle


\section{Introduction}

Let $G$ be a finite group and $\Irr(G)$ the set of its ordinary irreducible
characters. Frobenius showed the classical formula
$$\sum_{\chi\in\Irr(G)}\chi(1)^2=|G|$$
for the dimensions of the irreducible complex representations of $G$.
No analogue of this equation is known in the modular setting, that is, for the
set $\IBr(G)$ of irreducible $p$-Brauer characters of $G$, when $p$ is a prime
dividing $|G|$. In 2005, Willems \cite{Wi05} put forward a conjecture giving a
lower bound in terms of the prime-to-$p$ part $|G|_{p'}$ of the group order:

\begin{conjA}[Willems (2005)]   \label{conj:W}
 Let $G$ be a finite group and $p$ be a prime. Then
 $$|G|_{p'}\le\sum_{\vhi\in\IBr(G)}\vhi(1)^2.$$
\end{conjA}

Willems \cite{Wi05} points out that his conjecture holds for groups with cyclic
Sylow $p$-subgroups as well as for $p$-solvable groups, and he proves it for
groups of Lie type in defining characteristic~$p$. Here we show:

\begin{thmA}   \label{thm:main}
 Willems' Conjecture~\ref{conj:W} holds for all groups for the prime $p=2$.
\end{thmA}

Our proof relies on a reduction of (a strengthening of) the conjecture by
Tong-Viet \cite[Prop.~1.1]{TV19} to the case of quasi-simple groups. He shows
that Conjecture~\ref{conj:W} holds for all finite groups at the prime $p$ if the
following conjecture on $p$-Brauer characters is true for all quasi-simple
groups:

\begin{conjA}[Tong-Viet (2019)]   \label{conj:TV}
 Let $G$ be a finite quasi-simple group with centre $Z(G)$ of $p'$-order. Then
 $$|G/Z(G)|_{p'}\le\sum_{\vhi\in\IBr(G|\theta)}\vhi(1)^2$$
 for all faithful characters $\theta\in\Irr(Z(G))$.
\end{conjA}

Here, for $N\unlhd G$ and $\theta\in\IBr(N)$ we denote by $\IBr(G|\theta)$ the
set of irreducible Brauer characters of $G$ above $\theta$. Also, recall that
a finite group $G$ is \emph{quasi-simple} if $G$ is perfect and $G/Z(G)$ is
simple.   \par
Tong-Viet \cite[Prop.~2.1]{TV19} has verified his conjecture for the
finitely many quasi-simple groups of Lie type with exceptional Schur multiplier,
as well as for all covering groups of sporadic simple groups and for alternating
groups of small degree. We invoke the classification of finite simple groups to
deal with the general case, obtaining a complete answer at least when $p=2$. It
will turn out that in many cases the inequality in Conjecture~\ref{conj:TV} is
already satisfied with one single suitable Brauer character, like for symmetric
groups in characteristic~2 or many groups of Lie type in defining
characteristic, but on the other hand there are groups for which the degrees of
a large number of characters have to be taken into consideration; there is no
absolute upper bound on the number of necessary characters even in the class of
quasi-simple groups.   \par
To show that there are sufficiently many such characters, we are led to derive
lower bounds for the number of conjugacy classes of finite groups of Lie type
containing regular semisimple elements (Proposition~\ref{prop:reg elt}) which
may be of independent interest. As an application we obtain that any simple
group of Lie type in characteristic~$p$ has at least two conjugacy classes of
length divisible by the $p$-part of the group order
(Corollary~\ref{cor:sambale}); this was recently used by Sambale \cite{Sa20}.
\medskip

After collecting some basic observations, we study the finite simple groups
according to their classification, starting with the groups of Lie type. In
Section~\ref{sec:defchar} we suppose that $p$ is the  defining prime; here our
results are only partial for certain types. In Section~\ref{sec:crosschar} we
show Conjecture~\ref{conj:TV} in the non-defining characteristic situation for
all cases. Finally, in Section~\ref{sec:alt} we verify Conjecture~\ref{conj:TV}
for alternating groups for the prime $p=2$.

\section{Preliminary results}

\begin{prop}   \label{prop:normal}
 Let $G$ be a finite group and $N\unlhd G$ with $G/N$ solvable. Assume
 Conjecture~\ref{conj:W} holds for $G$ at the prime $p$. Then it also holds for
 $N$ at $p$.
\end{prop}

\begin{proof}
Let $N\le M\le G$ be a maximal (normal) subgroup. Then $G/M$ is cyclic of
prime order. Using Clifford theory and an easy counting argument we conclude
that Conjecture~\ref{conj:W} holds for $M$. The general statement thus follows
by induction over a composition series of $G/N$. 
\end{proof}

We will use the following consequence of a result of Kiyota and Wada:

\begin{lem}   \label{lem:cyclic}
 Conjecture~\ref{conj:TV} holds for quasi-simple groups with cyclic Sylow
 $p$-subgroups.
\end{lem}

\begin{proof}
It was shown by Kiyota and Wada \cite[Prop.~4.7]{KW93} (see also
\cite[Prop.~3.1]{HW07}) that for every $p$-block $B$ of a finite group $G$ with
a cyclic defect group $D$ we have
$$\dim B\le |D|\sum_{\vhi\in\IBr(B)}\vhi(1)^2.$$
Now assume that $G$ is quasi-simple with a cyclic Sylow $p$-subgroup and centre
$Z(G)$ of $p'$-order. Let $\theta\in\Irr(Z(G))$ be faithful. By
\cite[Thm.~(9.2)]{Na98}, for example, there is a
union $\cB$ of $p$-blocks of $G$ with $\Irr(G|\theta)=\bigcup_{B\in\cB}\Irr(B)$
and $\IBr(G|\theta)=\bigcup_{B\in\cB}\IBr(B)$. Let $d$ be the maximal order of
a defect group of any $B\in\cB$. Then, using the above inequality we have
$$|G/Z(G)|_{p'}\le \frac{1}{d}|G/Z(G)|
  =\frac{1}{d}\sum_{\chi\in \Irr(G|\theta)}\chi(1)^2
  =\frac{1}{d}\sum_{B\in\cB}\dim B \le \sum_{\vhi\in\IBr(G|\theta)}\vhi(1)^2,$$
as claimed by Conjecture~\ref{conj:TV}.
\end{proof}

In many cases, we will make use of characters of \emph{$p$-defect zero}, that
is, irreducible characters $\chi$ of a finite group $G$ such that $\chi(1)$
contains the full $p$-part of the group order $|G|$. It is a basic result of
Brauer that these remain irreducible under $p$-modular reduction and thus
furnish irreducible $p$-Brauer characters of the same degree.

\section{Groups of Lie type in defining characteristic}   \label{sec:defchar}

Throughout this section $\bG$ denotes a simple, simply connected linear
algebraic group over the algebraic closure of a finite field and $F:\bG\to\bG$ a
Steinberg endomorphism with finite group of fixed point $\bG^F$.
Any simple group of Lie type can then be obtained as $\bG^F/Z(\bG^F)$ for a
suitable such $\bG$, except for $\tw2F_4(2)'$, for which
Conjecture~\ref{conj:TV} already shown in \cite[Prop.~2.1]{TV19}.

In this section we consider the case where $p$ is the defining characteristic
of $\bG$. Here, Willems observed that the Steinberg character (of $p$-defect
zero) already has large enough degree for Conjecture~\ref{conj:W} to hold. Note,
however, that this result does not imply Conjecture~\ref{conj:TV} for these
groups, since generally there will be more than one block of positive defect.
Here, substantially more work is needed. In fact, it seems that not enough is
currently known about large degree irreducible modular characters in this
situation to derive a complete answer unless $p=2$. We will use the theory of
highest weight representations; a basic introduction can be found for example in
\cite[\S16]{MT}.

\begin{rem}   \label{rem:St}
 Let $\bG$ be simply connected and $F:\bG\to\bG$ a Frobenius endomorphism with
 respect to an $\FF_p$-structure. Recall that any irreducible representation
 of $\bG^F$ over $\overline{\FF_p}$ is the restriction of a highest weight
 representation of $\bG$ with $p$-restricted highest weight. Let
 $\{\psi_i\mid i\in I\}$ be a set of irreducible representations of $\bG$ (and
 hence of $\bG^F$) with restricted highest weights, all lying over the same
 central character $\theta$ of $\bG$ (and hence over the same central character
 of $\bG^F$). Further assume that
 $\sum_{i\in I}\psi_i(1)^2\ge |\bG^F:(\ker\theta)^F|_{p'}$. Let $\chi$ be the
 Steinberg representation of $\bG$ and let $r\ge1$. Then by Steinberg's tensor
 product theorem (see, e.g., \cite[Thm.~16.12]{MT}), all
 $\psi_{i,r}:=\psi_i\otimes\chi^{(1)}\otimes\cdots\otimes\chi^{(r-1)}$ are
 irreducible representations of $\bG^{F^r}$, where $\chi^{(j)}$ denotes the
 $j$th Frobenius twist of $\chi$. Since we have $\chi(1)^2>|\bG^F|_{p'}$,
 $$\sum_{i\in I}\psi_{i,r}(1)^2\ge |\bG^{F^r}:(\ker\theta)^{F^r}|_{p'}.$$
 Thus, if Conjecture~\ref{conj:TV} holds for $\bG^F$, it holds for all
 $\bG^{F^r}$, $r\ge1$.
 In proving Conjecture~\ref{conj:TV} for $\bG^{F^r}$ in the defining
 characteristic, we may therefore restrict attention to the groups defined over
 the prime field whenever convenient.
\end{rem}

Let us first deal with some small rank cases.

\begin{prop}   \label{prop:A2}
 Conjecture~\ref{conj:TV} holds for all central quotients of the groups
 $\SL_2(q)$, $\SL_3(q)$, $\SU_3(q)$, $\SL_4(q)$, $\SU_4(q)$ and  $\Sp_4(q)$ in
 defining characteristic.
\end{prop}

\begin{proof}
Let $G$ be as in the statement. By \cite[Exmp.~2.3(a)]{Wi05} we may restrict our
investigations to the case that $Z(G)\ne1$ and by Remark~\ref{rem:St} we may
assume that $q=p$. For $G=\SL_2(q)$, we may hence assume that $q=p$ is odd. Here
$G$ has a $p$-restricted faithful irreducible Brauer character of degree
$p-1$, and $(p-1)^2\ge(p^2-1)/2=|G/Z(G)|_{p'}$.
\par
Next assume $G=\SL_3(p)$ or $\SU_3(p)$ with $Z(G)$ non-trivial of order~3. For
$1\le i\le p-1$ let $\chi_i$ be the irreducible $p$-Brauer character of $G$ with
(restricted) highest weight $(i,i-1)$. We obtain a lower bound on $\chi_i(1)$ by
adding up the dimensions of certain weight spaces in the corresponding highest
weight module $L(i,i-1)$. According to the result of Premet \cite[Thm.~1]{Pr87},
all weights in the corresponding characteristic~0 irreducible highest weight
module also occur in $L(i,i-1)$. But these are exactly the weights subdominant
to $(i,i-1)$. It is now an easy exercise to determine all these weights and to
find the lengths of their orbits under the Weyl group. This shows that
$\chi_i(1)=\dim L(i,i-1)\ge 3i^2$. Hence
$\sum_{i=1}^{p-1}\chi_i(1)^2\ge (p^2-1)(p^3+1)/3$ by a straightforward
calculation. Thus we conclude, since all characters $\chi_i$ lie above a fixed
faithful character of $Z(\SL_3)$ (see e.g.\ \cite[App.~A.2]{Lue01}).
\par
Next, assume that $G=\Sp_4(p)$. Since $|Z(G)|=(2,p-1)$ we may assume that $p$
is odd. Note that a module with restricted highest weight $(i,j)$ is faithful on
$Z(G)$ if
and only if $j$ is odd. Arguing as in the previous case, one sees that a lower
bound for the dimension of the restricted module with highest weight $(i,2j+1)$
is $2i^2+(8j+6)i+4(j+1)^2$, and the sum of the squares of this over all
$i=0,\ldots,p-1$ and $j=0,\ldots,(p-3)/2$ is larger than $p^6$ and thus larger
than $|G/Z(G)|_{p'}=(p^2-1)(p^4-1)/2$.   \par
Finally, for $G=\SL_4(p)$ or $\SU_4(p)$, lower bounds for the weight space
dimensions in restricted representations with highest weight $(m_1,m_2,m_3)$
with $m_i\le p-1$ and $m_1+2m_2+3m_3\equiv1\pmod4$ (for faithful representations
of $G$), respectively $m_1+2m_2+3m_3\equiv2\pmod4$ (for those with a central
subgroup of order~2 in the kernel), yield a sum of squares at least 
$(p^2-1)(p^3\pm1)(p^4-1)/(4,p\pm1)$.
\end{proof} 

\begin{prop}   \label{prop:gen def}
 Let $G$ be quasi-simple of Lie type in characteristic~$p$, but not a spin,
 half spin or symplectic group, nor of type $\tw{(2)}A_{n-1}$ with $n\le 5$.
 Then Conjecture~\ref{conj:TV} holds for $G$ at the prime~$p$.
\end{prop}

\begin{proof}
We may assume that $G$ is not an exceptional covering group of its simple
quotient, since these have centre of order divisible by the characteristic~$p$
(see e.g. \cite[Tab.~24.3]{MT}). Thus, $G$ is a central factor group of a simply
connected group of Lie type. For those, as pointed out by Willems
\cite[Exmp.~2.3(a)]{Wi05}, Conjecture~\ref{conj:TV} holds for groups with
trivial centre. It remains to discuss the groups with non-trivial centre. In
particular we need not worry about Suzuki- and Ree groups. Let $G=\bG^F$, with
$\bG$ simple of simply connected type and $F$ a Frobenius endomorphism of $\bG$
with respect to an $\FF_q$-structure.

First assume that $G=E_6(q)$ or $\tw2E_6(q)$ with $|Z(G)|=3$. Then $G$ has a
subgroup $H=F_4(q)\times Z(G)$, where the first factor is the centraliser of a
graph automorphism of $G$. Let $\psi\in\IBr(H)$ be the Steinberg character
of $F_4(q)$ times a faithful character of $Z(G)$. Then $\psi(1)=q^{24}$.
Clearly, $G$ has to have a faithful irreducible $p$-Brauer character $\vhi$ of
at least that degree. But then $\vhi(1)^2\ge q^{48}>|G|_{p'}$. Next, assume
$G=E_7(q)$ with $|Z(G)|=2$. Here, consider the subgroup $H=E_6(q)\times Z(G)$,
and let $\psi\in\IBr(H)$ be the Steinberg character of $E_6(q)$ times the
faithful linear character of $Z(G)$. This shows that $G$ has an irreducible
$p$-Brauer character $\vhi$ with $\vhi(1)^2\ge\psi(1)^2=q^{72}>|G|_{p'}$.

Next, let $G=\SL_n(q)$. Then $G$ has a subgroup $H=\SL_{n-1}(q)\times Z(G)$.
The Steinberg character of $\SL_{n-1}(q)$ times a linear character of $Z(G)$
has degree $q^{(n-1)(n-2)/2}$, so $G$ has an irreducible $p$-Brauer character
of degree at least that large, while $|G|_{p'}\le q^{(n-1)(n+2)/2}$. Thus, we
are done when $n\ge6$.

For $G=\SU_n(q)$ with $n\ge6$, we again obtain an irreducible $p$-Brauer
character from a Steinberg character $\psi$ of a subgroup
$\SU_{n-1}(q)\times Z(G)$, with $\psi(1)^2=q^{(n-1)(n-2)}$. Now using that
$(q^k-1)(q^{k+1}+1)\le q^{2k+1}$ for all $k\ge2$ we find
$$|G|_{p'}=\prod_{k=2}^n(q^k-(-1)^k)\le q^{(n-1)(n+2)/2},$$
which again is smaller than $\psi(1)^2$ when $n\ge6$. 

Since we exclude the spin, half spin and symplectic groups by assumption,
it only remains to discuss the groups $G=\SO_{2n}^\pm(q)'$, $q$ odd, $n\ge4$.
These contain a subgroup $H=\SO_{2n-1}(q)'Z(G)$, and the product of the
Steinberg character of $\SO_{2n-1}(q)'$ with the faithful character of $Z(G)$
yields an irreducible Brauer character of degree $q^{(n-1)^2}$. Thus $G$ itself
also has a faithful character of at least that degree, and its square is larger
than $|G|_{p'}$ when $n\ge4$.
\end{proof}

\begin{prop}   \label{prop:class def}
 Let $G$ be quasi-simple of classical Lie type in characteristic~$p=2$. Then
 Conjecture~\ref{conj:TV} holds for $G$ at the prime $2$.
\end{prop}

\begin{proof}
By the results of Propositions~\ref{prop:A2} and~\ref{prop:gen def} we only
need to consider the groups $\SL_n(q)$, $\SU_n(q)$ for $n=5$ and $q=2^f$, as
the spin, half spin and symplectic groups have trivial centre in
characteristic~2.
For $\SL_5(q)$ the tensor product of the natural module with $f-1$ twists of the
1024-dimensional restricted Steinberg module yields an irreducible
representation of sufficiently large degree when $f\ge3$. For $f\le2$ the
centre of $\SL_5(2^f)$ is trivial. For the unitary groups, we can argue
similarly for $f\ge3$. For $f=2$ we take the tensor product of the Steinberg
character with the exterior square of the natural representation, and for
$\SU_5(2)$ the centre is trivial.
\end{proof}

We now state some further partial results.

\begin{lem}   \label{lem:Spn}
 Let $G=\Sp_{2n}(q)$ with $q=p^f$. Then Conjecture~\ref{conj:TV} holds for $G$
 at $p$ if either $n\ge5$ or $f\ge2$.
\end{lem}

\begin{proof}
Recall that we only need to consider the faithful characters of $G$.
The symplectic group $\Sp_{2n+2}(q)$ has a subsystem subgroup
$\Sp_{2n}(q)\times\Sp_2(q)$, and its image in $H:=\PSp_{2n+2}(q)$ is a central
product $G_1:=\Sp_{2n}(q)\circ\Sp_2(q)$ containing a subgroup $\Sp_{2n}(q)$.
Now the Steinberg character of $H$ of degree $q^{(n+1)^2}$ has $p$-defect zero,
so yields an irreducible $p$-Brauer character of $H$ of that degree.
Hence the subgroup $G_1$ has to have a faithful irreducible Brauer character
of degree at least $q^{(n+1)^2}/|H:G_1|$. Since the faithful irreducible Brauer
characters of $\Sp_2(q)$ have degree at most $q-1$, $G$ must have a faithful
irreducible Brauer character of degree at least $q^{(n+1)^2}/((q-1)|H:G_1|)$. But
the square of this is at least $|G|_{p'}$ whenever $n\ge5$. When $n\le4$ but
$f\ge2$, we may take the tensor product of that character at $q=p$ with twists
of the Steinberg character.
\end{proof}

Thus, with Proposition~\ref{prop:A2} among symplectic groups $\Sp_{2n}(q)$ only
the faithful block for $n=3,4$ and $q=p$ odd remains open.

\begin{prop}
 Let $G$ be a covering group of a simple orthogonal group of rank $n\ge3$
 over $\FF_q$, and assume that $q=p^f$ with $f\ge2$. Then
 Conjecture~\ref{conj:TV} holds for all covering groups of $G$ at $p$ unless
 possibly when $G=\Spin_8^-(q)$, where it holds when $f\ge3$.
\end{prop}

\begin{proof}
By Proposition~\ref{prop:class def} it remains to discuss the spin groups and
half spin groups for odd~$q$. Also, $\Spin_{4n}^+(q)$ has non-cyclic centre, so
we need not consider it. First assume that
$G=\Spin_{2n}^+(p)$ with $n\ge3$ odd. Let $H$ be the preimage in $G$ of the
stabiliser $\GL_n(p)$ of a totally isotropic subspace for $\SO_{2n}^+(p)$. As
$n$ is odd, $\PSL_n(p)$ has odd Schur multiplier, so the Steinberg character of
$\SL_n(p)$ extends to a faithful $p$-Brauer character $\psi$ of $H$, of degree
$p^{n(n-1)/2}$. Thus, $G$ has a faithful irreducible character of at least that
degree. Then, by Steinberg's tensor product theorem, $\Spin_{2n}^+(q)$ has a
faithful irreducible Brauer character of degree at least $(q/p)^{n^2}\psi(1)$,
which is larger than $|G|_{p'}$ when $f\ge2$. Now $\Spin_{2n+1}(p)$ contains a
subgroup $\Spin_{2n}^+(p)$ whose centre lies in the centre of $\Spin_{2n+1}(p)$.
As the Steinberg character of $\Spin_{2n+1}(p)$ has degree $p^{n^2}$, this
shows that $\Spin_{2n+1}(q)$ has a character as claimed (where still $n$ is odd).
Unless $n=3$ the same argument also applies to $\HSpin_{2n+2}^+(q)$,
$\Spin_{2n+2}^-(q)$, $\Spin_{2n+3}(q)$ and $\Spin_{2n+4}^-(q)$.

The group $\HSpin_8^+(q)$ is isomorphic to $\SO_8^+(q)$ by the triality
automorphism and thus by Proposition~\ref{prop:gen def} has a faithful
Brauer character of degree $q^9$. Using this, we also obtain our claim for
$\Spin_9(q)$ and $\Spin_{10}^-(q)$. 
\end{proof}

\section{The non-defining characteristic case}   \label{sec:crosschar}
We now turn to groups of Lie type $G$ in cross characteristic. That is,
we assume that $p$ is not the defining characteristic of $G$. Here, Willems
\cite[Thm.~3.1]{Wi05} obtained certain asymptotic results on
Conjecture~\ref{conj:W}, which were improved upon for several families of groups
in the thesis of Maslowski \cite{Ma05}. Nonetheless, both of these fall short of
proving Conjecture~\ref{conj:TV} in the case at hand.

We will again argue using suitable characters of $p$-defect zero, as considered
by Willems \cite{Wi88}, but in order to obtain complete results we need to show
that there exist sufficiently many of these.

Let $\bG$ be a simple algebraic group and $F:\bG\to\bG$ a Steinberg
endomorphism. The regular semisimple elements are known to be dense in $\bG$, so
we certainly expect a large proportion of regular semisimple classes in the
finite group $\bG^F$. Here, motivated by our application in character theory,
we quantify this expectation by giving a (rather weak) lower bound for their
number.

\begin{prop}   \label{prop:reg elt}
 Let $(\bG,F)$ be as above, with
 $$\bG^F\ne \SU_4(2),\ \Sp_4(2),\ \Sp_6(2),\ \OO_8^+(2),\  \OO_8^-(2),\ 
   \SU_3(3).$$
 Then a lower bound for the number $n_\reg(T)$ of conjugacy classes of regular
 semisimple conjugacy classes meeting certain maximal tori $T$ of $\bG^F$ is
 bounded below as given in Table~$\ref{tab:bnd}$ for classical groups, and in
 Table~$\ref{tab:tori exc}$ for exceptional type groups.
\end{prop}

\begin{proof}
The proof is independent of the isogeny type of $\bG$. Let $q$ be the absolute
value of the eigenvalues of $F$ on the character group of an $F$-stable maximal
torus of $\bG$. We will use the following
well-known result of Zsigmondy: for every integer $e\ge3$ and prime power $q$
there exists a prime $z_e(q)$ dividing $q^e-1$ but no $q^m-1$ for $1\le m<e$,
unless $(e,q)=(6,2)$. Since $q$ has multiplicative order $e$ modulo $z_e(q)$, we
see that $z_e(q)\ge e+1$.
\par
For each group $G=\bG^F$ of classical type we have given in Table~\ref{tab:bnd}
(the order of) two maximal tori $T\le G$. These two tori have been chosen such
that the greatest common divisor of their orders is exactly the order of the
centre of the simply connected group $G_\SC$ of that type, which is also the
order of the commutator factor group of the adjoint group of that type.
Furthermore, the image of $T$ in $G_\SC/Z(G_\SC)$ is cyclic, except for the
first two tori listed for type $D_n$. For each torus the table also gives an
integer $e$ such that $|T|$ is divisible by a Zsigmondy prime $r:=z_e(q)$,
whenever that exists, that is, if $e\ge3$ and $(e,q)\ne(6,2)$. In all cases
listed in the table, the description of Sylow $p$-subgroups in
\cite[Thm.~25.14]{MT} together with the order formula for $G$ (see
\cite[Tab.~24.1]{MT}) shows that the Sylow $r$-subgroups of $G$ are cyclic.
Moreover, by the parametrisation of maximal tori via conjugacy classes of the
Weyl group (see \cite[Prop.~25.3]{MT}), the conjugates of $T$
are the only maximal tori of $G$ containing elements of order $r$, except for
the torus of order $(q^{n-1}-1)(q+1)$ in types $B_n,C_n$. In particular, apart
from that latter case, all elements of $T$ of order divisible by $r$ are
necessarily regular. Now $T$ certainly contains at least $|T|(r-1)/r$ such
elements.
Since $r\ge e+1$, we thus find at least $|T|e/(e+1)$ regular elements in~$T$.
In the case of the torus $T$ of order $(q^{n-1}-1)(q+1)$, all products of
elements of order $r$ by an element of order at least~3 in the factor of order
$q+1$ are regular, and thus there exist at least $(q^{n-1}-1)(q-1)e/(e+1)$
regular elements in $T$.

The fusion of elements in any maximal torus is controlled by its normaliser,
so the number $n_\reg(T)$ of regular conjugacy classes in $G$ with
representatives in $T$ is at least $|N_G(T):T|^{-1}$ times the number of
regular elements in $T$. These are exactly the entries in the last column of
Table~\ref{tab:bnd}.

\begin{table}[ht]
\caption{Lower bounds in classical groups}   \label{tab:bnd}
$$\begin{array}{c|cccc}
 G& |T|& e& |N_G(T):T|& n_\reg(T)\ge\cr
\hline\hline
 A_1(q)& q+1& 2& 2& (q-1)/2\\
       & q-1& 1& 2& (q-3)/2\\
\hline
 A_{n-1}(q),\ n\ge3& (q^n-1)/(q-1)& n& n& |T|/(n+1)\\
                   & q^{n-1}-1& n-1& n-1& |T|/n\\
\hline
     \tw2A_{n-1}(q),& (q^n+1)/(q+1)& 2n& n& 2|T|/(2n+1)\\
  3\le n\equiv1\,(2)& q^{n-1}-1& n-1& n-1& |T|/n\\
\hline
     \tw2A_{n-1}(q),& q^{n-1}+1& 2n-2& n-1& 2|T|/(2n-1)\\
  4\le n\equiv0\,(2)& (q^n-1)/(q+1)& n& n& |T|/(n+1)\\
\hline
 B_n(q),C_n(q),\ n\ge2& q^n+1& 2n& 2n& |T|/(2n+1)\\
                  n=2:& q^2-1&  2&  4& (q-1)(q-2)/4\\
   4\le n\equiv0\,(2):& (q^{n-1}-1)(q+1)& n-1& 4n-4& (q^{n-1}-1)(q-1)/(4n)\\
        n\equiv1\,(2):& q^n-1& n& 2n& |T|/(2n+2)\\
\hline
  D_n(q),\ n\ge4& (q^{n-1}+1)(q+1)& 2n-2& 2n-2& |T|/(2n-1)\\
  n\equiv0\,(2):& (q^{n-1}-1)(q-1)& n-1& 2n-2& |T|/(2n)\\
  n\equiv1\,(2):& q^n-1& n& n& |T|/(n+1)\\
\hline
 \tw2D_n(q),\ n\ge4& q^n+1& 2n& n& 2|T|/(2n+1)\\
                   & (q^{n-1}+1)(q-1)& 2n-2& 2n-2& |T|/(2n-1)\\
\end{array}$$
\end{table}

We now discuss the cases in Table~\ref{tab:bnd} when there do not always exist
Zsigmondy primes. For $G=A_1(q)$ a maximal torus of order $q\pm1$ contains at
least $q\pm1-\gcd(2,q-1)$ regular elements and hence representatives from
$(q\pm1-\gcd(2,q-1))/2$ regular conjugacy classes. For $G=A_2(q)$, it can be
seen by direct
calculation that a maximal torus of order $q^2-1$ intersects $(q^2-q)/2\ge |T|/3$
regular classes, as stated in the table. For $G=\tw2A_2(q)$, a maximal torus of
order $q^2-1$ meets $(q+1)(q-2)/2$ regular classes, which is smaller than
$|T|/3$ only when $q=2,3$. In the first case, $G$ is solvable, the second was
excluded.  For $G=B_2(q)$ a maximal torus of order $q^2-1$ meets $(q-1)(q-2)/4$
regular semisimple classes, as stated in Table~\ref{tab:bnd}.

To complete the discussion of Zsigmondy exceptions, we finally consider the
cases when $(e,q)=(6,2)$ in Table~\ref{tab:bnd}. This concerns the following
groups:
$$\begin{array}{c|cccccccc}
 G& \SL_6(2)& \SL_7(2)& \SU_4(2)& \SU_6(2)& \SU_7(2)& \Sp_6(2)& \OO_8^+(2)& \OO_8^-(2)\\
\hline
 |T|& 63& 63& 9& 21& 63& 9& 27& 9\\
 n_\reg(T)& 9& 9& 2& 3& 9& 1& 3& 1\\
\end{array}$$
The numbers $n_\reg(T)$ can be read off from the known character tables.
For $\SL_6(2)$, $\SL_7(2)$, $\SU_6(2)$ and $\SU_7(2)$ this agrees with the
numbers given in Table~\ref{tab:bnd}, while the other groups are listed as
exceptions.

We now turn to the groups of exceptional type, for which the arguments are
very similar. In Table~\ref{tab:tori exc} for each type we give a maximal torus
$T$, respectively two tori for $G$ of type $E_7$. In all cases, $T$ is cyclic
by \cite[Thm.~25.14]{MT} and for all groups different from $E_7(q)$, all
elements of $T$ of order not dividing $|Z(G_\SC)|$ are regular by the order
formula for $G$ (see \cite[Tab.~24.1]{MT}). For $E_7(q)$, the maximal tori
are Coxeter tori of maximal rank subgroups of type $\tw2A_7(q)$, $A_7(q)$
respectively, and all of their elements not lying in the subgroup of order
$q\pm1$ are regular.

\begin{table}[ht]
\caption{Lower bounds in exceptional groups}   \label{tab:tori exc}
$$\begin{array}{l|ccc}
 G& |T|& |N_G(T):T|& n_\reg(T)\cr
\hline
 \!\tw2B_2(q^2),\ q^2\ge8& \Phi_8''& \ 4& (|T|-1)/4\cr
 \!\!^2G_2(q^2),\ q^2\ge27& \Phi_{12}''& \ 6& (|T|-1)/6\cr
 G_2(q),\,q\equiv1\,(3)& \Phi_6& \ 6& (|T|-1)/6\cr
 \hphantom{G_2(q),\ }q\not\equiv1\,(3)& \Phi_3& \ 6& (|T|-1)/6\cr
 \!\tw3D_4(q)& \Phi_{12}& \ 4& (|T|-1)/4\cr
 \!\tw2F_4(q^2),\ q^2\ge8& \Phi_{24}''& 12& (|T|-1)/12\cr
 F_4(q)& \Phi_{12}& 12& (|T|-1)/12\cr
 E_6(q)& \Phi_9& \ \,9& (|T|-(3,q-1))/9\cr
 \!\tw2E_6(q)& \Phi_{18}& \ \,9& (|T|-(3,q+1))/9\cr
 E_7(q)& \Phi_2\Phi_{14}& 14& (q^7-q)/14\cr
       & \Phi_1\Phi_7& 14& (q^7-q)/14\cr
 E_8(q)& \Phi_{24}& 24& (|T|-1)/24\cr
\end{array}$$
Here $\Phi_8''=q^2+\sqrt{2}q+1$, $\Phi_{12}''=q^2+\sqrt{3}q+1$,
$\Phi_{24}''=q^4+\sqrt{2}q^3+q^2+\sqrt{2}q+1$.
\end{table}

Each regular element in $T$ is conjugate to $|N_G(T):T|$ elements of $T$. Since
torus normalisers control fusion of their elements, we obtain the stated lower
bounds for the number $n_\reg(T)$ of $G$-conjugacy classes of regular semisimple
elements meeting $T$.
\end{proof}

Note that the total number of semisimple classes for $G$ of simply connected
type was shown by Steinberg to be equal to $q^l$, where $l$ is the rank of
$\bG$. With a lot more effort it would be possible to derive estimates for the
number of regular semisimple classes that are asymptotically much closer to
$q^l$.

We note the following easy consequence, which has been used in the proof of
a recent result by Sambale \cite[Thm.~10]{Sa20}:

\begin{cor}   \label{cor:sambale}
 Let $G$ be a simple group of Lie type in characteristic $p$. Then there exist
 at least two conjugacy classes of elements of $G$ with centraliser order prime
 to $p$.
\end{cor}

\begin{proof}
Let $\bG$ be simple of simply connected type with a Steinberg endomorphism $F$
such that $G=\bG^F/Z(\bG^F)$. This is possible unless $G=\tw2F_4(2)'$, for which
the claim is easily verified. Let $x\in\bG^F$ be regular semisimple. Then its
centraliser is a maximal torus, of order prime to $p$. Thus it suffices to
show that $\bG^F$ has at least two conjugacy classes of regular semisimple
elements which do not have the same image in $G$. For exceptional type groups
$G$ it is immediate that the numbers $n_\reg(T)$ in Table~\ref{tab:tori exc}
are all strictly bigger than $|Z(\bG^F)|$, and so we obtain at least two
regular semisimple $G$-classes. For the groups of classical type, we take
regular semisimple classes coming from the two different types of maximal tori
given in Table~\ref{tab:bnd}. For the exceptions in
Proposition~\ref{prop:reg elt} the claim is again readily verified.
\end{proof}

\begin{thm}   \label{thm:class cross}
 Conjecture~\ref{conj:TV} holds for quasi-simple groups of Lie type for
 non-defining primes $p$.
\end{thm}

\begin{proof}
By \cite[Prop.~2.1]{TV19} the claim holds for the exceptional covering groups,
thus we may assume that $G$ is a central quotient $G=\bG^F/Z$ of a quasi-simple
group $\bG^F$ of simply connected Lie type. We will construct sufficiently
many irreducible Deligne--Lusztig characters of $G$ of $p$-defect zero  using
Proposition~\ref{prop:reg elt}. First, the assertion is readily checked from
the known Brauer character tables for the six groups listed as exceptions in
that result. So we may assume that $G$ is not a covering group of one of
those. As discussed in the proof of Proposition~\ref{prop:reg elt}, for $G$ of
classical type, the two maximal tori of $\bG^F$ in Table~\ref{tab:bnd} are such
that their images in $\bG^F/Z(\bG^F)$ have coprime orders. Thus, we may choose
at least one of them, say $T$, with $|T|_p=|Z(G)|_p$ for our given prime~$p$.
The same applies to groups of type $E_7$. For the remaining groups of
exceptional type, all elements of $T\setminus Z(\bG^F)$ are regular and the
corresponding Sylow subgroups of $G$ are cyclic. Since by
Lemma~\ref{lem:cyclic} we need not consider primes $p$ for which Sylow
subgroups are cyclic, we again have that $|T|_p=|Z(G)|_p$ for the tori $T$ in
Table~\ref{tab:tori exc}.
\par
Let $\bT\le\bG$ be an $F$-stable maximal torus with $\bT^F=T$, and let
$(\bT^*,\bG^*,F)$ be dual to $\bT,\bG,F)$ (see \cite[Def.~1.5.17]{GM20}), so
$T^*:=\bT^{*F}$ has the same order as $T$. To each regular element $s\in T^*$,
up to $G^*$-conjugation, there exists an irreducible Deligne--Lusztig character
$\pm R_s$ of~$\bG^F$ of degree $|\bG^F:T|_{q'}$ (see \cite[Def.~2.5.17 and
Thm.~2.2.12]{GM20}). Hence, by the choice of $T$, $R_s$ is of central
$p$-defect, that is, the defect of $\chi$ equals the $p$-part of the centre
of $\bG^F$. So the $p$-modular reduction of $R_s$ is irreducible (see e.g.
\cite[Thm.~(9.13)]{Na98}). Now by the character formula
\cite[Prop.~2.2.18]{GM20}, $R_s$ is a faithful character of $G=\bG^F/Z$ with
$Z=\ker(\theta)\cap Z(\bG^F)$, where $s$ is in duality with $(\bT,\theta)$, for
$\theta\in\Irr(T)$. By construction the Zsigmondy prime $r$ for $T$ occurring in
the proof of Proposition~\ref{prop:reg elt} does not divide
$|\bG^{*F}:[\bG^{*F},\bG^{*F}]|$. So we may choose a system of representatives
$R$ for the cosets of $T^*\cap[\bG^{*F},\bG^{*F}]$ in $T^*$ consisting of
elements of order prime to $r$. Thus, if $s\in T^*$ has order divisible by~$r$
then so has $st$, for $t\in R$. It follows that the regular elements in $T^*$ of
order divisible by $r$ are distributed equally across the various cosets of
$T^*\cap[\bG^{*F},\bG^{*F}]$ in $T^*$. Thus, the irreducible Deligne--Lusztig
characters of central $p$-defect constructed above contribute
$$n_\reg(T)/|Z|\cdot|\bG^F:T|_{q'}^2$$
to the sum in Conjecture~\ref{conj:TV}. So, we need to see that
$$n_\reg\ge \frac{|G|_{q}\cdot |T|^2}{|G|_{q'}\cdot|P|},\eqno{(*)}$$
where $P$ denotes a Sylow $p$-subgroup of $G$.

For this, we estimate $|P|$. Since, as pointed out before, we only need to
consider
primes $p$ for which Sylow $p$-subgroups of $G$ are non-cyclic, $p$ divides
at least two (not necessarily distinct) cyclotomic factors occurring in the
order formula for $G$ (see \cite[Thm.~25.14]{MT}). Note that if $p>2$ and $q$
has order $d$ modulo~$p$, then $p\ge d+1$, while for $p=2$, at least one of
$q-1$, $q+1$ is divisible at least by~4. Again first assume that $\bG$ is of
classical type, say of rank~$n$. A Sylow $p$-subgroup $P$ of $G$ has order at
least~$p^2$. As the Weyl group $W$ of $\bG$ contains a symmetric group $\fS_n$,
this shows $|P|\ge(n+1)^2$ if $p$ does not divide the order of $W$. If $p\le n$,
then $|P|\ge p^{\lfloor n/p\rfloor+p}$, which is still at least $(n+1)^2$ unless
$p=2$. But for $p=2$ we have $|P|\ge 2^{\lfloor n/2\rfloor+4}$, and it follows
that $|P|\ge(n+1)^2$ in all cases.

For groups of exceptional type, from the order formula for $G$ (see
\cite[Tab.~24.1]{MT}) it is easy to check that the order of the smallest
non-cyclic Sylow $p$-subgroups is at least~25, and even at least~121 when
$G=E_8(q)$.

Now comparing $(*)$ with the right hand side of the formula in
Conjecture~\ref{conj:TV} we see that our conclusion follows except when $G$ is
one of $\SL_3(2),\Sp_4(2)$, for which the claim is easily checked directly,
or $G=\Sp_{2n}(q)$ or $\Spin_{2n+1}(q)$ with $n=4$ and $q\le5$, $n=6$ and
$q\le4$, $n=8,10$ with $q\le3$ or $(n,q)=(12,2)$, and in all cases, $p$ divides
the order $q^n+1$ of the first torus from Table~\ref{tab:bnd}. Among these, the
only cases in which Sylow $p$-subgroups of $G$ are non-cyclic are for 
$$(n,q,p)\in\{(6,2,5),(6,3,5),(6,4,17),(10,2,5),(10,3,5),(12,2,17)\}.$$
In all these cases except for $(n,p,q)=(6,2,5)$, the desired inequality holds
when we use the exact order of $P$ (which is at least $p^3$ here). Finally, for
$\Sp_{12}(2)$, the ordinary character table allows us to see that the sum of
squares of 5-defect zero characters exceeds the bound.
\end{proof}

\section{Alternating groups}   \label{sec:alt}

It remains to discuss the covering groups of simple alternating groups. At
present we only see how to treat the prime $p=2$.

\begin{thm}   \label{thm:alt}
 Let $G$ be a covering group of the alternating group $\fA_n$, $n\ge5$. Then
 Conjecture~\ref{conj:TV} holds for $G$ at the prime~$2$.
\end{thm}

\begin{proof}
By Proposition~\ref{prop:normal} it is sufficient to consider the covering
groups of the symmetric groups. The exceptional covering groups of $\fA_6$ and
$\fA_7$ are easily checked, so we may assume that $n\ge8$ and so $Z(G)=1$, that
is $G=\fS_n$. We will show that for large enough $n$ there exists a single
irreducible 2-Brauer character whose degree satisfies the desired inequality.
In fact, as for the groups of Lie type, it is the 2-modular reduction of an
ordinary character, namely, a faithful character of the 2-fold cover $2.\fS_n$,
but not of 2-defect zero.   \par
For $l\ge1$ consider the two partitions
$$\begin{aligned}
  p_1&=(4l-3,4l-7,\ldots,1)\vdash n_{l,1}:=l(2l-1),\\
  p_2&=(4l-1,4l-5,\ldots,3)\vdash n_{l,2}:=l(2l+1).
\end{aligned}$$
These label irreducible characters $\chi_l^i$ of the 2-fold cover
$2.\fS_{n_{l,i}}$ of $\fS_{n_{l,i}}$. Their degree is given by the analogue of
the hook length formula (see \cite[Thm.~2.8]{Fa18}): for a partition
$\la=(\la_1,\la_2,\ldots,\la_m)\vdash n$ with distinct parts $\la_i$, the degree
of the corresponding spin character of $2.\fS_n$ equals
$$2^{\lfloor(n-m)/2\rfloor}\frac{n!}{\prod_i\la_i!}\cdot
  \prod_{i<j}\frac{\la_i-\la_j}{\la_i+\la_j}.$$
From this one easily computes that
$$\frac{\chi_{l+1}^1(1)}{\chi_l^1(1)}
  =2^{4l-1}\binom{n_{l+1,1}}{4l+1}\binom{4l-1}{2l}^{-1}\quad\text{and}\quad
 \frac{\chi_{l+1}^2(1)}{\chi_l^2(1)}
  =2^{4l+1}\binom{n_{l+1,2}}{4l+3}\binom{4l+1}{2l+1}^{-1}.$$
We claim that for $l\ge8$ we have
$$\chi_l^1(1)^2 \ge (n_{l,2}-1)!_{2'}\qquad\text{and}\qquad
  \chi_l^2(1)^2 \ge (n_{l+1,1}-1)!_{2'}\,.\eqno{(*)}
$$
Assume that the inequalities have already been shown up to $l$. Now 
$$\binom{4l-1}{2l}=\frac{1}{2}\binom{4l}{2l}\le 2^{4l-1}/\sqrt{4l}\quad\text{and}\quad
  \binom{4l+1}{2l+1}=\frac{1}{2}\binom{4l+2}{2l+1}\le 2^{4l+1}/\sqrt{4l+2},$$
by the standard estimate for the middle binomial coefficient, and so
$$\chi_{l+1}^1(1)^2\ge 4l\binom{n_{l+1,1}}{4l+1}^2(n_{l,2}-1)!_{2'},
  \qquad\chi_{l+1}^2(1)^2\ge(4l+2)\binom{n_{l+1,2}}{4l+3}^2(n_{l+1,1}-1)!_{2'}$$
by our inductive assumption. Further, with $n:=n_{l+1,1}$ we have
$$\begin{aligned}
  \binom{n}{4l+1}
  &\ge c_1\Big(\frac{n}{4l+1}\Big)^{4l+1}\Big(\frac{n}{n-4l-1}\Big)^{n-4l-1}
  \sqrt{\frac{n}{2\pi(4l+1)((n-4l-1)}}\\
  &\ge c_1\sqrt\frac{1}{2\pi(4l+1)}\Big(1+\frac{4l+1}{n-4l-1}\Big)^{n-4l-1}
    \Big(\frac{n}{4l+1}\Big)^{4l+1}\\
  &\ge c_2\sqrt\frac{1}{2\pi(4l+1)}\ e^{4l+1}\Big(\frac{n}{4l+1}\Big)^{4l+1}
\end{aligned}$$
for some constants $c_1,c_2\ge0.5$ independent of $l$ by a well-known estimate
for binomial coefficients. On the other hand by Stirling's formula
$$\begin{aligned}
  \frac{(n+2l+1)!}{(n-2l-2)!}
  &\le c_3 \frac{\sqrt{2\pi(n+2l+1)}}{\sqrt{2\pi(n-2l-2)}}\cdot
  \frac{(n+2l+1)^{n+2l+1}\,e^{n-2l-2}}{(n-2l-2)^{n-2l-2}\,e^{n+2l-1}}\\
  &=c_4 e^{-4l-3}\Big(1+\frac{4l+3}{n-2l-2}\Big)^{n-2l-2}(n+2l+1)^{4l+3}\\
  &\le c_4(n+2l+1)^{4l+3},
\end{aligned}$$
where again $c_3,c_4\le2$ are independent of $l$. Putting things together,
this shows that
$$\frac{\chi_{l+1}^1(1)^2}{(n_{l+1,2}-1)!_{2'}}
  \ge \frac{c_2^2}{c_4} 4l\frac{e^{8l+2}}{2\pi(4l+1)}\Big(\frac{n}{4l+1}\Big)^{8l+2}
  \Big(\frac{1}{n+2l+1}\Big)^{4l+3}\frac{(n_{l+1,2}-1)!_2}{(n_{l,2}-1)!_2}.
$$
Now, using that $n=n_{l+1,1}=(l+1)(2l+1)$ we find
$$\frac{n^2}{(4l+1)^2(n+2l+1)}\ge \frac{1}{8},$$
and since $(n_{l+1,2}-1)!_2/(n_{l,2}-1)!_2\ge 2^{4l+3}/(4l+3)$, 
$$\frac{\chi_{l+1}^1(1)^2}{(n_{l+1,2}-1)!_{2'}}
  \ge c_5 \frac{4l}{(4l+1)(4l+3)} \Big(\frac{1}{n+2l+1}\Big)^{2}
  \Big(\frac{e^2}{4}\Big)^{4l-1}.$$
The last term on the right hand side clearly dominates when $l\to\infty$ and so
the left hand side is eventually bigger than~1. A more
precise estimate and checking the first 50 values by computer then completes
the proof of our claim for $\chi_l^1$. A very similar computation shows~(*)
for $\chi_l^2$.
\par
The desired assertion then follows: by a result of Fayers \cite[Thm.~3.3]{Fa18},
both $\chi_l^1$ and $\chi_l^2$ remain irreducible modulo~2. So for all
$n\ge n_{l+1,i}$, with $l\ge8$, there exists an irreducible 2-Brauer character
of $\fS_n$ of degree at least $\chi_{l+1}^i(1)$. By (*), $\chi_{l+1}^1(1)$
satisfies the inequality from Conjecture~\ref{conj:TV} for $\fS_n$ for all
$n_{l+1,1}\le n\le n_{l+1,1}+2l+1=n_{l+1,2}-1$, and similarly,
$\chi_{l+1}^2(1)$ satisfies our inequality for $\fS_n$ whenever
$n_{l+1,2}\le n\le n_{l+2,1}-1$. Finally, for $l<8$, that is, for $n<120$ it is
easily checked using the
criterion in \cite[Thm.~3.3]{Fa18} that there always is an irreducible
2-Brauer character with the desired property, labelled by a sum of one of
our partitions $\la_{l,1}$ or $\la_{l,2}$ and a suitable Carter partition.
\end{proof}

It should be noted that we need to use both characters $\chi_l^1$ and $\chi_l^2$
for our approach to work. 

For odd primes, the combinatorics seem more daunting, and furthermore, there
are two quite different cases to consider, corresponding to faithful and
non-faithful characters of $2.\fS_n$. Note that by Lemma~\ref{lem:cyclic}, we may
assume that $p\le n/2$ for $\fS_n$ and $2.\fS_n$.
\medskip

Our Main Theorem~\ref{thm:main} now follows by combining
Proposition~\ref{prop:class def}, Theorem~\ref{thm:class cross} and
Theorem~\ref{thm:alt} with the result of Tong-Viet \cite[Prop.~2.1]{TV19} for
sporadic groups.


\end{document}